\documentclass[12pt]{article}

\usepackage{amsmath}
\usepackage{amssymb}
\usepackage{amsthm}

\newtheorem{theorem}{Theorem}

\newtheorem{lemma}{Lemma}

\newtheorem{question}{Question}

\title{Numbers with three close factorizations and central lattice points on hyperbolas}
\author{Tsz Ho Chan}
\date{}

\begin{document}
\maketitle
\begin{abstract}
In this paper, we continue the study of three close factorizations of an integer and correct a mistake of a previous result. This turns out to be related to lattice points close to the center point $(\sqrt{N}, \sqrt{N})$ of the hyperbola $x y = N$. We establish optimal lower bounds for $L^1$-distance between these lattice points and the center. We also give some good examples based on polynomials and Pell equations more systematically.
\end{abstract}

\section{Introduction and main results}

Factorization is a fundamental study on integers. In this paper, we continue the study of close factorizations of an integer such as
\[
3950100 = 1881 \cdot 2100 = 1890 \cdot 2090 = 1900 \cdot 2079.
\]
Suppose a positive integer $N$ admits $k$  different factorizations:
\begin{equation} \label{factor1}
N = A B = (A + a_1) (B - b_1) = \cdots = (A + a_{k-1}) (B - b_{k-1})
\end{equation}
for some positive integers $A$, $B$, $a_1 < a_2 < \cdots < a_{k-1}$, and $b_1 < b_2 < \cdots < b_{k-1}$. One can ask how small $a_{k-1}$ and $b_{k-1}$ can be in terms of $N$ and $k$. Rephrasing it another way, one can ask how big $A$ and $B$ can be in terms of $k$, $a_{k-1}$ and $b_{k-1}$. In \cite{C1}, the author proved the following result for numbers with three close factorizations.
\begin{theorem} \label{thm0}
Suppose $N$ satisfies \eqref{factor1} with $k = 3$. Then
\[
A, B \le \frac{1}{4} \max(a_2, b_2) (\max(a_2, b_2) - 1)^2
\]
whenever $\max(a_2, b_2) \ge 4$.
\end{theorem}
\noindent However, it contains a slight error as the example
\[
N = 72 = 6 \cdot 12 = 8 \cdot 9 = 9 \cdot 8 \; \text{ with } \; A = 6, \; B = 12, \; a_1 = 2, \; b_1 = 3, \; a_2 = 3, \; b_2 = 4
\]
shows that $12 \le \frac{1}{4} \cdot 4 \cdot (4 - 1)^2$ fails. We shall give a correct and more conceptual proof of the following modified version.
\begin{theorem} \label{thm0.5}
Suppose $N$ satisfies \eqref{factor1} with $k = 3$. Then
\[
A, B \le \frac{1}{4} \max(a_2, b_2) (\max(a_2, b_2) - 1)^2
\]
whenever $\max(a_2, b_2) \ge 5$. Note: The proof is much simpler in the range $\max(a_2, b_2) \ge 7$.
\end{theorem}
\noindent It is closely related to the following result with a shorter and more intuitive proof than \cite{C1}.
\begin{theorem} \label{thm1}
Suppose $N$ satisfies \eqref{factor1} with $k = 3$ with $B \le \sqrt{N} \le A$. Then
\[
A \le \frac{1}{4} a_2 (a_2 - 3) (a_2 + 1).
\]
Moreover, the above equality can be attained if and only if $a_2 > 3$ is odd and
\[
A = \frac{a_2 (a_2 + 1) (a_2 - 3)}{4}, \; B = \frac{(a_2 - 2) (a_2 - 1)^2}{4}, \; a_1 = \frac{a_2 + 1}{2}, \; b_1 = \frac{a_2 - 1}{2}, \; b_2 = a_2 - 2.
\]
Note: $a_{2}$ is the largest among $a_1, a_2, b_1, b_2$ as
\begin{equation} \label{ab}
A B = (A + a_i)(B - b_i) \; \text{ implies } \; a_i B = b_i A + a_i b_i > b_i A \; \text{ and } \; a_i > \frac{A}{B} b_i \ge b_i.
\end{equation}
\end{theorem}

Equivalently, one can interpret $k$ close factorizations of $N$ as $k$ close lattice points on the hyperbola $x y = N$. Previously, Cilleruelo and Jim\'{e}nez-Urroz \cite{CJ1, CJ2} and Granville and Jim\'{e}nez-Urroz \cite{GJ} did some study on this and related problem. Suppose $(A, B)$, $(A + a_1, B - b_1)$ and $(A + a_2, B - b_2)$ with $1 \le a_1 < a_2$ and $1 \le b_1 < b_2$ are three lattice points on the hyperbola $x y = N$. Then Theorem \ref{thm0} or \ref{thm1} tells us that
\begin{equation} \label{lattice1}
\max(a_2, b_2) > (4 \max(A, B))^{1/3} \ge 2^{2/3} N^{1/6}.
\end{equation}
On the other hand, there are infinitely many $N$'s (namely the example in Theorem \ref{thm1}) such that the hyperbola $x y = N$ contains three such lattice points with
\begin{equation} \label{lattice2}
a_2 - 2 \le (4 B)^{1/3} < (4 A)^{1/3} \; \; \text{ or } \; \; b_2 < a_2 < 2^{2/3} N^{1/6} + 2. 
\end{equation}
Now, observe that
\[
A - B = \frac{a_2 (a_2 + 1)(a_2 - 3)}{4} - \frac{(a_2 - 2) (a_2 - 1)^2}{4} = \frac{a_2^2 - 4 a_2 + 1}{2} > \frac{a_2^2}{4} > \frac{N^{1/3}}{2}
\]
when $a_2 > 8$. Hence, either $A - \sqrt{N}$ or $\sqrt{N} - B$ is greater than $\frac{N^{1/3}}{4}$. So, the above three-close-lattice-points examples cannot be too close to the center point $(\sqrt{N}, \sqrt{N})$ of the hyperbola $x y = N$. This phenomenon was, for example, noticed by Cilleruelo and Jim\'{e}nez-Urroz \cite{CJ2}. It leads to the following question.
\begin{question}
Suppose $(x_1, y_1)$, $(x_2, y_2)$, $(x_3, y_3)$ are three positive lattice points on $x y = N$. What can we say about the maximal $L^1$-distance
\[
\max_{i} |x_i - \sqrt{N}| + |y_i - \sqrt{N}| \; \; \Bigl( = \max_{i} |x_i - y_i| \text{ as } x_i \le \sqrt{N} \le \frac{N}{x_i}  \text{ or } x_i \ge \sqrt{N} \ge \frac{N}{x_i}  \Bigr)
\]
between the lattice points and the center $(\sqrt{N}, \sqrt{N})$?
\end{question}
We have the following optimal answer which reiterates \cite{CJ2} on the contrast between the exponent $1/4$ and the exponent $1/6$ in \eqref{lattice1} and \eqref{lattice2}.
\begin{theorem} \label{thm2}
Suppose $x y = N$ has three positive lattice points $(x_1, y_1)$, $(x_2, y_2)$ and $(x_3, y_3)$. Then, for $N \ge 6$,
\[
\max_{i} |x_i - y_i| \ge \Big\lfloor 2 N^{1/4} + \frac{1}{2 N^{1/4} - 1} \Big\rfloor.
\]
Moreover, the lower bound is best possible. Here $\lfloor x \rfloor$ is the largest integer not exceeding $x$.
\end{theorem}
\noindent Similar to the polynomial idea in \cite{GJ}, based on \eqref{poly-xy} and \eqref{poly-N} with $K = 99$, we have
\[
N = 99990000 = 9900 \cdot 10100 = 9999 \cdot 10000 = 10100 \cdot 9900
\]
with
\[
10100 - 9900 = 200 \; \; \text{ and } \; \; 2 N^{1/4} + \frac{1}{2 N^{1/4} - 1} = 200.00397 \ldots .
\]

However, when one looks at the extremal examples for Theorem \ref{thm2}, not all of the three lattice points are to the right or to the left of the center point. This motivates the next question.
\begin{question}
Suppose $(x_1, y_1)$, $(x_2, y_2)$, $(x_3, y_3)$ are three lattice points on $x y = N$ with $\sqrt{N} < x_1 < x_2 < x_3$. What can we say about $\max_{i} |x_i - y_i|$?
\end{question}
\begin{theorem} \label{thm3}
Suppose $x y = N$ has three lattice points $(x_1, y_1)$, $(x_2, y_2)$ and $(x_3, y_3)$ with $\sqrt{N} < x_1 < x_2 < x_3$. Then, for $N \ge 2^{20}$,
\[
\max_{i} |x_i - y_i| \ge \lfloor 2 \sqrt{2} N^{1/4} + 1/2 \rfloor.
\]
Moreover, the lower bound is best possible. Note: The lower bound $\max_{i} |x_i - y_i| > 2 \sqrt{2} N^{1/4} - 2$ is much easier to establish.
\end{theorem}
\noindent For instance, we have the following example as a consequence:
\[
N = 5997600 = 2450 \cdot 2448 = 2499 \cdot 2400 = 2520 \cdot 2380
\]
with
\[
|2520 - 2380| = 140 \; \; \text{ and } \; \; 2 \sqrt{2} N^{1/4} = 139.9714\ldots .
\]
Our construction using solutions of Pell equations may be compared to that in \cite{CJ2} using continued fraction convergents. We leave the interested readers to investigate optimal lower bounds for four or more lattice points.

\bigskip

The paper is organized as follows. First, we will establish Theorem \ref{thm1} and use it to prove Theorem \ref{thm0.5} when $\max(a_2, b_2) \ge 7$ using Theorem \ref{thm2}. Then we will prove the entire Theorem \ref{thm0.5} via some finer analysis of lattice points close to the center point. Next, we will prove Theorem \ref{thm2} and show some polynomial examples. Then we will show Theorem \ref{thm3} with the easier lower bound $2 \sqrt{2} N^{1/4} - 2$ which illustrates the essence of the method. Lastly, we will prove the entire Theorem \ref{thm3} and provide examples from solutions of a Pell equation and a polynomial.

\section{Proof of Theorem \ref{thm1}}

\begin{proof}
Expanding $A B = (A + a_i) (B - b_i)$ and simplifying, we get $a_i B - b_i A = a_i b_i$. Dividing through by $a_i A$, we have
\begin{equation} \label{basic1}
\frac{B}{A} - \frac{b_i}{a_i} = \frac{b_i}{A}.
\end{equation}
Applying \eqref{basic1} with $i = 1, 2$ and subtracting the two equations, we obtain
\begin{equation} \label{basic2}
\frac{b_1}{a_1} - \frac{b_2}{a_2} = \frac{b_2 - b_1}{A} \; \; \text{ or } \; \; A = \frac{a_2 a_1 (b_2 - b_1)}{D}
\end{equation}
where
\[
D := a_2 b_1 - b_2 a_1 = (a_2 - a_1) b_1 - (b_2 - b_1) a_1 > 0.
\]
Analogously, dividing $a_i B - b_i A = a_i b_i$ through by $b_i B$ and subtracting the two equations with $i = 1, 2$, we also have
\begin{equation} \label{basic3}
\frac{a_2}{b_2} - \frac{a_1}{b_1} = \frac{a_2 - a_1}{B} \; \; \text{ or } \; \; B = \frac{b_2 b_1 (a_2 - a_1)}{D}.
\end{equation}
Now, let $C = a_2 = (1 + \alpha) a_1$ for some real number $\alpha > 0$. Then
\[
a_1 = \frac{C}{1 + \alpha} \; \; \text{ and } \; \; a_2 - a_1 = \alpha a_1 = \frac{\alpha C}{1 + \alpha}.
\]
From the definition of $D$ and $b_1 \le a_1 - 1$ via \eqref{ab},
\[
b_2 - b_1 = \frac{(a_2 - a_1) b_1 - D}{a_1} \le \Bigl[ \frac{\alpha C}{1 + \alpha}  \Bigl( \frac{C}{1 + \alpha} - 1 \Bigr) - D \Bigr] \frac{1 + \alpha}{C} = \frac{\alpha C}{1 + \alpha} - \alpha - \frac{D (1 + \alpha)}{C}.
\]
Hence, it follows from \eqref{basic2} that
\begin{align*}
A =& \frac{a_2 a_1 (b_2 - b_1)}{D} \le \frac{C}{D} \cdot \frac{C}{1 + \alpha} \cdot \Bigl( \frac{\alpha C}{1 + \alpha} - \alpha - \frac{D (1 + \alpha)}{C} \Bigr) \\
=& \frac{\alpha C^3}{D (1 + \alpha)^2} - \frac{\alpha C}{D (1 + \alpha)} - C =: t(\alpha).
\end{align*}
The derivative
\[
t'(\alpha) = \frac{(1 - \alpha) C^3}{D (1 + \alpha)^3} - \frac{C^2}{D (1 + \alpha)^2}
\]
is positive when $0 < \alpha < \frac{C - 1}{C + 1}$ and is negative when $\alpha > \frac{C - 1}{C + 1}$. Therefore, $t(\alpha)$ has a maximum at $\alpha = \frac{C - 1}{C + 1}$ and
\[
A \le \frac{\frac{C - 1}{C + 1} \cdot C^3}{(1 + \frac{C - 1}{C + 1})^2} - \frac{\frac{C - 1}{C + 1} \cdot C}{(1 + \frac{C - 1}{C + 1})} - C = \frac{C (C - 3) (C + 1)}{4}.
\]
as $D \ge 1$. One can see that the above equality can be achieved if and only if $D = 1$, $a_1 = \frac{C}{1 + \frac{C - 1}{C+1}} = \frac{C+1}{2}$ and $b_1 = a_1 - 1 = \frac{C - 1}{2}$ are integers (i.e., $C$ is odd). Then $b_2 = \frac{a_2 b_1 - D}{a_1} = C - 2$. Hence,
\[
a_2 = 2K + 1, \; a_1 = K + 1, \; b_2 = 2K - 1, \; b_1 = K, \; A = (2K+1)(K+1)(K - 1), \; B = (2 K - 1) K^2
\]
and we have Theorem \ref{thm1}.
\end{proof}

\section{Proof of Theorem \ref{thm0.5} for $\max(a_2, b_2) \ge 7$}

\begin{proof}
Recall $N = A B = (A + a_1)(B - b_1) = (A + a_2)(B - b_2)$ with $1 \le a_1 < a_2$ and $1 \le b_1 < b_2$. If $A \ge B$, then Theorem \ref{thm0} follows from Theorem \ref{thm1}. So, let us suppose $B > A$ from now on.

\bigskip

Case 1: $B - b_2 \ge A + a_2$. Then we can rearrange the factorizations as
\[
N = (B - b_2) (A + a_2) = (B - b_2 + (b_2 - b_1)) (A + a_2 - (a_2 - a_1)) = (B - b_2 + (b_2)) (A + a_2 - (a_2)).
\]
Applying Theorem \ref{thm1} with $B - b_2$, $A +a_2$, $b_2 - b_1$, $a_2 - a_1$, $b_2$, $a_2$ in place of $A$, $B$, $a_1$, $b_1$, $a_2$, $b_2$ respectively, we get
\[
B - b_2 \le \frac{1}{4} b_2 (b_2 - 3) (b_2 + 1)
\]
which gives Theorem \ref{thm0} as $\frac{1}{4} x (x - 3) (x + 1) + x = \frac{1}{4} x (x - 1)^2$. This identity shows the connection between the upper bounds in Theorems \ref{thm0.5} and \ref{thm1}.

\bigskip

Case 2: $B - b_2 < A + a_2$. Then we have $B > \sqrt{N} > B - b_2$, $A < \sqrt{N} < A + a_2$ and $a_2 + b_2  > B - A$. By Theorem \ref{thm2},
\[
\max \{ B - A, |(B - b_1) - (A + a_1)|, (A + a_2) - (B - b_2) \} \ge \lfloor 2 N^{1/4} \rfloor.
\]
If $(a_1 + b_1) - (B - A) \ge 0$, then $(a_2 + b_2) - (B - A) > (a_1 + b_1) - (B - A) \ge 0$. If $(a_1 + b_1) - (B - A) < 0$, then $B - A > (B - A) - (a_1 + b_1) > 0$. In either case, the middle term $|(B - b_1) - (A + a_1)|$ cannot be the maximum. Therefore,
\begin{equation} \label{basic-ineq}
\max \{ B - A, (a_2 + b_2) - (B - A) \} \ge \lfloor 2 (A B)^{1/4} \rfloor \ge \lfloor 2 \sqrt{A} \rfloor
\end{equation}
Since $a_2 + b_2 > B - A > 0$, the above inequality implies
\[
2 \max(a_2, b_2) - 1 \ge a_2 + b_2 - 1 > 2 \sqrt{A} - 1 \; \text{ or } \; A \le \max(a_2, b_2)^2 - 1,
\]
and
\[
B < A + a_2 + b_2 \le \max(a_2, b_2)^2 + 2 \max(a_2, b_2) - 1 \le \frac{1}{4} \max(a_2, b_2) (\max(a_2, b_2) - 1)^2
\]
for $\max(a_2, b_2) \ge 7$ as $\frac{1}{4} x (x - 1)^2 - (x^2 + 2x - 2) = \frac{1}{4} x (x - 7)(x + 1) + 2$ is positive and increasing when $x \ge 7$. This gives the easier version of Theorem \ref{thm0.5}.
\end{proof}

\section{Proof of Theorem \ref{thm0.5} for $\max(a_2, b_2) \ge 5$}

\begin{proof}
To get down to $\max(a_2, b_2) \ge 5$, we need to do more careful analysis on the close factorizations in case 2 from the previous section.

\bigskip

\noindent Subcase 1: $B - b_1 > A + a_2$. Then one can focus on the following three close factorizations
\[
N = (A + a_2) (B - b_2) = (B - b_1) (A + a_1) = B A
\]
instead. Since $B > B - b_1 > A + a_2 > \sqrt{N}$, we can apply Theorem \ref{thm1} and get
\[
A + a_2 \le \frac{1}{4} (B - A - a_2) (B - A - a_2 - 3) (B - A - a_2 + 1) < \frac{1}{4} b_2 (b_2 - 3)(b_2 + 1)
\]
as $B - A - a_2 < b_2$ and $x (x - 3) (x + 1)$ is an increasing function for $x \ge 3$. Hence,
\[
B < A + a_2 + b_2 < \frac{1}{4} b_2 (b_2 - 3)(b_2 + 1) + b_2
\]
and we have Theorem \ref{thm0.5}.

\bigskip

\noindent Subcase 2: $B - b_1 < A + a_2 < B$. Then one can focus on the following three close factorizations
\[
N = (B - b_2)(A + a_2) = (B - b_1)(A + a_1) = (A + a_2)(B - b_2)
\]
instead. This is of the form
\begin{equation} \label{til}
N = \tilde{A} \tilde{B} = (\tilde{A} + \tilde{a}_1) (\tilde{B} - \tilde{b}_1) = (\tilde{A} + \tilde{a}_2) (\tilde{B} - \tilde{b}_2)
\end{equation}
with 
\[
\tilde{A} = B - b_2, \; \tilde{B} = A + a_2, \; \tilde{a}_1 = b_2 - b_1, \; \tilde{b}_1 = a_2 - a_1, \; \tilde{a}_2 = \tilde{b}_2 = a_2 + b_2 + A - B.
\]
One has $\tilde{a}_2 + \tilde{b}_2 = 2 (\tilde{B} - \tilde{A})$. We also have $\tilde{B} = A + a_2 > B - b_2 = \tilde{A}$ and $\tilde{B} - \tilde{b}_2 = B - b_2 < A + a_2 = \tilde{A} + \tilde{a}_2$. Hence, we are in case 2 above and can apply inequality \eqref{basic-ineq} to get
\[
b_2 > A - B + a_2 + b_2 = \tilde{B} - \tilde{A} \ge \lfloor 2 \sqrt{\tilde{A}} \rfloor \ge 2 \sqrt{B - b_2} - 1
\]
Hence,
\[
B < \Bigl( \frac{b_2 + 1}{2} \Bigr)^2 + b_2
\]
which gives Theorem \ref{thm0.5} as $\frac{1}{4} x (x - 1)^2 - (\frac{x+1}{2})^2 - x = \frac{1}{4} (x + 1)(x^2 - 4 x - 1)$ is positive and increasing for $x \ge 5$.

\bigskip

\noindent Subcase 3: $B - b_1 < B \le A + a_2$. Then one can focus on the following three close factorizations
\[
N = A B = (A + a_1) (B - b_1) = B A
\]
instead. This is of the form
\[N = \tilde{A} \tilde{B} = (\tilde{A} + \tilde{a}_1) (\tilde{B} - \tilde{b}_1) = (\tilde{A} + \tilde{a}_2) (\tilde{B} - \tilde{b}_2)
\]
with
\[
\tilde{A} = A, \; \tilde{B} = B, \; \tilde{a}_1 = a_1, \; \tilde{b}_1 = b_1, \; \tilde{a}_2 = \tilde{b}_2 = B - A.
\]
One has $\tilde{B} - \tilde{b}_2 = A < B = \tilde{A} + \tilde{a}_2$ and $\tilde{a}_2 + \tilde{b}_2 = 2(\tilde{B} - \tilde{A})$. Hence, we are in case 2 above and can apply inequality \eqref{basic-ineq} to get
\[
a_2 \ge B - A = \tilde{B} - \tilde{A} \ge \lfloor 2 \sqrt{\tilde{A}} \rfloor \ge 2 \sqrt{A} - 1 \; \; \text{ or } \; \; A \le \Bigl(\frac{a_2 + 1}{2} \Bigr)^2.
\]
Hence,
\[
B \le A + a_2 \le \Bigl(\frac{a_2 + 1}{2} \Bigr)^2 + a_2
\]
which gives Theorem \ref{thm0.5} in the same way as subcase 2 above.

\bigskip

\noindent Subcase 4: $B - b_1 = A + a_2$. Since $(A + a_1) (B - b_1) = (A + a_2) (B - b_2)$, we also have $B - b_2 = A + a_1$. Therefore, $B - A = a_2 + b_1 = a_1 + b_2$. From $A B = (A +a_2)(B - b_2)$, we have
\[
a_2 (B - A) - (b_2 - a_2) A = a_2 b_2 \; \; \text{ and } \; \; (a_2 - b_2) B + b_2 (B - A) = a_2 b_2
\]
Substituting $B - A = a_1 + b_2$ and $B - A = a_2 + b_1$ to the above equations respectively, we get
\[
(b_2 - a_2) A = a_1 a_2 \; \; \text{ and } \; \; (b_2 - a_2) B = b_1 b_2.
\]
If $b_2 - a_2 \ge 2$, the above yields $A, B \le \frac{\max(a_2, b_2)^2}{2}$ which gives Theorem \ref{thm0.5} as $\frac{1}{4} x (x - 1)^2 - \frac{x^2}{2} = \frac{1}{4} x (x^2 - 4x + 1)$ is positive and increasing for $x \ge 4$. It remains to deal with the situation where $b_2 = a_2 + 1$. Then we have $A = a_1 a_2$ and $B = b_1 b_2$. Substituting these into $A B = (A + a_2)(B - b_2)$, we obtain $b_1 = a_1 + 1$. Hence,
\[
A = a_1 a_2 \; \;  \text{ and } \; \; B = (a_1 + 1)(a_2 + 1)
\]
with $1 \le a_1 < a_2$. The three close factorizations are
\[
a_1 a_2 \cdot (a_1 + 1)(a_2 + 1) = a_2 (a_1 + 1) \cdot a_1 (a_2 + 1) = a_1 (a_2 + 1) \cdot (a_1 + 1) a_2.
\]
We have $\max_i (a_i, b_i) = b_2 = a_2 + 1$ and $A < B = (a_1 + 1)(a_2 + 1) \le a_2 (a_2 + 1) \le \frac{1}{4} (a_2 + 1) (a_2 + 1 - 1)^2$ when $a_2 + 1 \ge 5$. This gives Theorem \ref{thm0.5} and shows why we need $\max(a_2, b_2) \ge 5$.

\bigskip

{\it Remark:} Even though the proof requires $\max(a_2, b_2) \ge 5$ for subcases 2 and 3 above, there is no difficulty with lowering it to $\max(a_2, b_2) \ge 4$. One can check that when $\max(a_2, b_2) = 4$, potential counterexamples can only happen when $B = 10$ since the bound from Theorem \ref{thm0.5} is $\frac{1}{4} \cdot 4 (4 - 1)^2 = 9$ while the bound from subcases 2 and 3 is $(\frac{4 + 1}{2})^2 + 4 = 10.25$. And we have $b_2 = 2, 3$ or $4$.

\bigskip

\noindent If $b_2 = 2$, then $b_1 = 1$. The integer $N$ satisfies $N = 10 A = 9 (A + a_1) = 8 (A + a_2)$. In particular, $2^3 \cdot 3^2 \cdot 5 = 360$ divides $N$ which implies $A \ge  36 > B$, a contradiction.

\bigskip

\noindent If $b_2 = 3$, then $b_1 = 1$ or $2$. The integer $N$ satisfies $N = 10 A = 9 (A + a_1) = 7 (A +a_2)$ or $N = 10 A = 8 (A + a_1) = 7 (A + a_2)$. Then, $630$ or $280$ divides $N$ which also gives a contradiction.

\bigskip

\noindent If $b_2 = 4$, then $b_1 = 1, 2$ or $3$. The integer $N$ satisfies $N = 10 A = 9 (A + a_1) = 6 (A + a_2)$ or $N = 10 A = 8 (A + a_1) = 6(A + a_2)$ or $N = 10 A = 7 (A + a_1) = 6 (A + a_2)$. The last two cases imply that $120$ or $280$ divides $N$ which gives a contradiction. In the first case, $A < B = 10$ is divisble by $9$. Hence, $A = 9$. Then $N = 90 = 9 \cdot 10 = 10 \cdot 9 = 15 \cdot 6$ as the only possible close factorizations. We have $a_2 = 6$, $b_2 = 4$ and $B = 10 \le \frac{1}{4} \cdot 6 (6 - 1)^2$. So, there is no counterexample to Theorem \ref{thm0.5} from these subcases, and the only counterexample comes from subcase 4 when $a_1 = 2$ and $a_2 = 3$.
\end{proof}

\section{Proof of Theorem \ref{thm2}}
\begin{proof}
To prove the theorem, it suffices to show that $\max_i |x_i - y_i| > 2 N^{1/4} - 1 + \frac{1}{2 N^{1/4} - 1}$ as $\lfloor 2 N^{1/4} + \frac{1}{2 N^{1/4} - 1} \rfloor$ is the smallest integer greater than $2 N^{1/4} - 1 + \frac{1}{2 N^{1/4} - 1}$. Without loss of generality, we focus on the situation where $x_1 < x_2 \le \sqrt{N}$. The case $x_3 > x_2 \ge \sqrt{N}$ is similar by symmetry as $y_3 < y_2 \le \sqrt{N}$. Then $y_1 > y_2 \ge \sqrt{N}$. If $|x_2 - y_2| > 2 N^{1/4} - 1/2$, then we are done as $N \ge 6$. Suppose $|x_2 - y_2| \le 2 N^{1/4} - 1/2$. Observe that
\[
y_2 - \sqrt{N} = \frac{N}{x_2} - \sqrt{N} = \frac{\sqrt{N}}{x_2} (\sqrt{N} - x_2) \ge \sqrt{N} - x_2.
\]
Hence,
\[
\sqrt{N} - x_2 \le \frac{(y_2 - \sqrt{N}) + (\sqrt{N} - x_2)}{2} \le N^{1/4} - \frac{1}{4} \; \; \text{ or } \; \; x_2 \ge \sqrt{N} - N^{1/4} + \frac{1}{4}.
\]
Say $x_1 = x_2 - b$ and $y_1 = y_2 + a$ for some integers $a, b > 0$. Then
\begin{equation} \label{key-eq}
(x_2 - b) (y_2 + a) = x_2 y_2 \; \; \text{ or } \; \; (a - b) x_2 = a b + b(y_2 - x_2) > 0. 
\end{equation}
Hence, $a > b$ and
\begin{align} \label{ineq}
\sqrt{N} - N^{1/4} + \frac{1}{4} \le x_2  \le& \Bigl( \frac{a + b}{2} \Bigr)^2 - \Bigl( \frac{a - b}{2} \Bigr)^2 + \Bigl( \frac{a + b}{2} \Bigr) (y_2 - x_2) - \Bigl( \frac{a - b}{2} \Bigr) (y_2 - x_2) \nonumber \\
\le& \Bigl( \frac{a + b}{2} \Bigr)^2 + \Bigl( \frac{a + b}{2} \Bigr) (y_2 - x_2) - \frac{3}{4}
\end{align}
Now, consider
\[
|x_1 - y_1| = |(x_2 - b) - (y_2 + a)| = (y_2 - x_2) + (a + b)
\]
Set $\Sigma := a + b$ and $\Delta := y_2 - x_2$. From \eqref{ineq}, we have
\[
4 \sqrt{N} - 4 N^{1/4} + 5 < \Sigma^2 + 2 \Sigma \Delta + 1 \le (\Sigma + \Delta)^2,
\]
and
\[
|x_1 - y_1| > \sqrt{(2 N^{1/4} - 1)^2 + 4} \ge 2 N^{1/4} - 1 + \frac{1}{2 N^{1/4} - 1}
\]
as $\sqrt{1 + x} \ge 1 + x / 4$ for $0 \le x \le 8$. This gives Theorem \ref{thm2}. The lower bound is optimal in view of the following examples: For any integer $K \ge 1$,
\begin{equation} \label{poly-xy}
\left\{
\begin{array}{ll}
x_1 = K^2 + K = K (K + 1), & y_1 = K^2 + 3K + 2 = (K + 1)(K + 2); \\
x_2 = K^2 + 2K = K (K + 2), & y_2 = K^2 + 2K + 1 = (K + 1)^2; \\
x_3 = K^2 + 2K + 1 = (K + 1)^2, & y_3 = K^2 + 2K = K (K + 2);
\end{array}
\right.
\end{equation}
and
\begin{equation} \label{poly-N}
N = x_1 y_1 = x_2 y_2 = x_3 y_3 = K (K + 1)^2 (K + 2).
\end{equation}
Note that $N^{1/4} < K + 1$ and
\begin{align*}
(K+1) - N^{1/4} =& (K + 1) \Bigl[ 1 - \Bigl( \frac{K (K + 2)}{(K+1)^2} \Bigr)^{1/4} \Bigr] \\
=& (K + 1) \Bigl[1 - \Bigl(1 -  \frac{1}{(K+1)^2} \Bigr)^{1/4} \Bigr] < \frac{1}{4(K+1)} + \frac{1}{4 (K+1)^3}
\end{align*}
as
\begin{equation} \label{lem0}
(1 - x)^{1/4} > 1 - \frac{x}{4} - \frac{x^2}{4} \; \; \text{ for } \; \; 0 < x \le \frac{1}{2}
\end{equation}
(which can easily be verified by raising both sides to the fourth power). Then one can check that
\[
2 N^{1/4} < \max_{i} |x_i - y_i| = 2 (K + 1) < 2 N^{1/4} + \frac{1}{2 N^{1/4} - 1},
\]
and we have the entire Theorem \ref{thm2}.
\end{proof}

{\it Note:} The idea behind the above example is based on \eqref{key-eq} and \eqref{ineq}. We want $y_2 - x_2$ and $a - b$ small. Hence, we set $a = K + 1$, $b = K$ and $y_2 - x_2 = 1$. Then equation \eqref{key-eq} yields $x_2 = (K+1) K + K (1) = K (K + 2)$ and $y_2 = x_2 + 1 = (K + 1)^2$. The lattice points $(x_1, y_1)$ and $(x_3, y_3)$ are simply $(x_2 - b, y_2 + a)$ and $(y_2, x_2)$ respectively. 

\bigskip

One can construct other (non-optimal) polynomial examples easily. For instance, say $a = 2K+1$, $b = 2K$ and $y_2 - x_2 = K$. Then equation \eqref{key-eq} gives $x_2 = 2K (3K + 1)$ and $y_2 = x_2 + K = 3K (2K + 1)$. The lattice points $(x_1, y_1)$ and $(x_3, y_3)$ are simply $(x_2 - b, y_2 + a) = ((2K)(3K), (2K + 1)(3K + 1))$ and $(y_2, x_2) = (3K (2K+1), 2K(3K + 1))$ respectively. Here $\max_i |x_i - y_i| = 5 K + 1$ which has size about $\frac{5 \sqrt{6}}{6} N^{1/4} \approx 2.04124 N^{1/4}$.

\section{Proof of Theorem \ref{thm3} with lower bound $2 \sqrt{2} N^{1/4} - 2$}

\begin{proof}
We revert to the strategy in Theorem \ref{thm1}. With $\sqrt{N} \le x_1 < x_2 < x_3$, suppose
\[
(x_1, y_1) = (A, B), \; \; (x_2, y_2) = (A + a_1, B - b_1), \; \; (x_3, y_3) = (A + a_2, B - b_2)
\]
for $1 \le a_1 < a_2$ and $1 \le b_1 < b_2$. The equality $N = A B = (A + a_i)(B - b_i)$ gives
\[
a_i B - b_i A = a_i b_i \; \; \text{ or } \; \;  (a_i - b_i) B = a_i b_i + b_i (A - B).
\]
Set $d_i = a_i - b_i$ for $i = 1, 2$. Note that $d_i B \ge a_i b_i > 0$. So, $d_1, d_2 > 0$. Since $D = a_2 b_1 - b_2 a_1 = (a_2 - b_2) b_1 - (a_1 - b_1) b_2 = d_2 b_1 - d_1 b_2 > 0$ and $b_2 > b_1$, we must have $d_2 > d_1 \ge 1$. We also have
\begin{equation} \label{dB}
\left\{ \begin{array}{l}
d_1 B = b_1 (b_1 + d_1) + b_1 (A - B), \\
d_2 B = b_2 (b_2 + d_2) + b_2 (A - B).
\end{array} \right.
\end{equation}

Note that $\max_i |x_i - y_i|$ occurs at $i = 3$ as $|x - y| = x - \frac{N}{x}$ is an increasing function of $x$ when $x \ge \sqrt{N}$. Hence, our goal is to study
\begin{equation} \label{x3-y3-1}
x_3 - y_3 = (A + a_2) - (B - b_2) = (A - B) + (2 b_2  + d_2).
\end{equation}
From \eqref{dB}, we have
\begin{equation} \label{A-B}
A - B = \frac{d_2 B}{b_2} - b_2 - d_2.
\end{equation}
Thus,
\begin{align} \label{x3-y3-2}
x_3 - y_3 =& \frac{d_2 B}{b_2} - b_2 - d_2 + (2 b_2 + d_2) = \frac{d_2 B}{b_2} + b_2 \nonumber \\
\ge& 2 \sqrt{\frac{d_2 B}{b_2} \cdot b_2} = 2 \sqrt{d_2 B} = 2 \sqrt{d_2} \Bigl(\frac{B}{A}\Bigr)^{1/4} N^{1/4}
\end{align}
by the inequality $x + y \ge 2 \sqrt{x y}$ and $N = A B$. If $A - B > 2 \sqrt{2} N^{1/4} - 8$, then we are done as $b_2 > b_1 \ge 1$, $d_2 > d_1 \ge 1$ and $x_3 - y_3 = A - B + 2 b_2 + d_2 \ge A - B + 6$. So, we suppose $A - B \le 2 \sqrt{2} N^{1/4} - 8$. Since $A \ge \sqrt{N}$,
\[
\frac{A - B}{A} \le \frac{2 \sqrt{2} N^{1/4} - 8}{A} \le \frac{2 \sqrt{2} N^{1/4} - 8}{N^{1/2}} \; \; \text{ or } \; \; \frac{B}{A} \ge 1 - \frac{2 \sqrt{2}}{N^{1/4}} + \frac{8}{N^{1/2}}.
\]
By \eqref{lem0}, we have
\[
x_3 - y_3 > 2 \sqrt{2} \Bigl(1 - \frac{\sqrt{2}}{2 N^{1/4}} \Bigr) N^{1/4} \ge 2 \sqrt{2} N^{1/4} - 2
\]
when $N \ge 1024$.
\end{proof}

\section{Proof of the entire Theorem \ref{thm3}}

We shall refine the analysis from the previous section and need the following lemmas.

\begin{lemma} \label{lem1}
Suppose $A \ge B > 0$ with $A - B \le C$. Then
\[
\frac{B}{A} \ge 1 - \frac{C}{A}.
\]
\end{lemma}

\begin{proof}
This follows easily by dividing $A - B \le C$ through by $A$ and rearranging terms.
\end{proof}

\begin{lemma} \label{lem2}
If two real numbers $0 < x \le y$ have difference $y - x \ge d$, then
\[
x + y \ge 2 \sqrt{x y} + \frac{d^2}{4 y}.
\]
\end{lemma}

\begin{proof}
Observe that
\[
\sqrt{y} - \sqrt{x} = \frac{y - x}{\sqrt{y} + \sqrt{x}} \ge \frac{d}{2 \sqrt{y}}.
\]
The lemma follows after squaring both sides and rearranging terms.
\end{proof}

\begin{proof}[Proof of Theorem \ref{thm3}]

To prove the theorem, it suffices to establish $x_3 - y_3 > 2 \sqrt{2} N^{1/4} - 1/2$ since $\lfloor 2 \sqrt{2} N^{1/4} + 1/2 \rfloor$ is the smallest integer greater than $2 \sqrt{2} N^{1/4} - 1/2$. If $A - B > 2 \sqrt{2} N^{1/4}$, then $x_3 - y_3 = A - B + 2 b_2 + d_2 > A - B + 1$ by \eqref{x3-y3-1} and we are done. Henceforth, we suppose $A - B \le 2 \sqrt{2} N^{1/4}$. From \eqref{x3-y3-2}, we have $x_3 - y_3 \ge 2 \sqrt{3} (\frac{B}{A})^{1/4} N^{1/4}$ if $d_2 \ge 3$. By Lemma \ref{lem1} and $A \ge \sqrt{N}$,
\[
x_3 - y_3 \ge 2 \sqrt{3} \Bigl(1 - \frac{2 \sqrt{2} N^{1/4}}{A} \Big)^{1/4} N^{1/4} \ge 2 \sqrt{3} \Bigl(1 - \frac{2 \sqrt{2}}{N^{1/4}} \Big)^{1/4} N^{1/4} \ge 2.9 N^{1/4} 
\]
when $N \ge 1024$. So, Theorem \ref{thm3} is true whenever $d_2 \ge 3$. 

\bigskip

It remains to deal with the case $d_2 = 2$ and $d_1 = 1$ since $d_2 > d_1 \ge 1$. Then $a_1 = b_1 + 1$, $a_2 = b_2 + 2$ and $D = 2 b_1 - b_2$. From \eqref{basic2} and \eqref{basic3}, we have
\begin{equation} \label{ABmore}
A = \frac{(b_2 + 2)(b_1 + 1)(b_2 - b_1)}{2 b_1 - b_2}, \; \; B = \frac{b_2 b_1 (b_2 - b_1 + 1)}{2 b_1 - b_2}, \; \; A - B = \frac{b_2^2 - 2 b_1^2 + 2 b_2 - 2 b_1}{2 b_1 - b_2}.
\end{equation}
Suppose $A - B \le N^{0.2}$. By Lemma \ref{lem1} and inequality \eqref{lem0},
\[
x_3 - y_3 \ge 2 \sqrt{2} \Bigl(1 - \frac{N^{0.2}}{A} \Big)^{1/4} N^{1/4} > 2 \sqrt{2} N^{1/4} - \frac{1}{\sqrt{2} N^{0.05}} - \frac{1}{\sqrt{2} N^{0.35}} > 2 \sqrt{2} N^{1/4} - \frac{1}{2}.
\]
for $N \ge 5000$. It remains to consider $N^{0.2} < A - B \le 2 \sqrt{2} N^{1/4}$. From \eqref{ABmore}, we have
\begin{equation} \label{almost}
b_2^2 - 2 b_1^2 + 2 b_2 - 2 b_1 > N^{0.2} (2 b_1 - b_2) \; \text{ or } \; b_2^2 - 2 b_1^2 > (2 N^{0.2} + 2) b_1 - (N^{0.2} + 2) b_2.
\end{equation}
Suppose $0 \ge b_2^2 - 2 b_1^2$. Then \eqref{almost} implies $b_2 \ge \frac{2 N^{0.2} + 2}{N^{0.2} + 2} b_1 > \sqrt{3} b_1$ when $N \ge 5000$. This yields the contradiction $0 > b_2^2 - 2 b_1^2 > 3 b_1^2 - 2 b_1^2 > 0$. Therefore, we must have $b_2^2 - 2 b_1^2 > 0$ or $b_2 > \sqrt{2} b_1$.

\bigskip

Suppose $\sqrt{2} b_1 < b_2 \le (\sqrt{2} + N^{-0.1}) b_1$ and $N \ge 2^{20}$. As  $b_2^2 - 2 b_1^2 = (b_2 - \sqrt{2} b_1)(b_2 + \sqrt{2} b_1)$, inequality \eqref{almost} gives
\[
\frac{(2 \sqrt{2} + N^{-0.1})}{N^{0.1}} > (2 - \sqrt{2}) N^{0.2} - N^{0.1} - (2 \sqrt{2} - 2) - 2 N^{-0.1}
\]
which is impossible as $(2 - \sqrt{2}) x^3 - x^2 - (2 \sqrt{2} - 2) x - 2 > 10$ for $N^{0.1} = x \ge 4$. Thus, $2 b_1 - 1 \ge b_2 > (\sqrt{2} + N^{-0.1}) b_1$. From \eqref{A-B}, we have
\begin{equation} \label{b1}
\sqrt{N} \le A \le (2 b_1 + 1) (b_1 + 1) (b_1 - 1) < 2 b_1^3 + b_1^2 \le 2.5 b_1^3 \; (\text{as } b_1 > 1) \; \text{ or } \; b_1 > \frac{N^{1/6}}{\sqrt[3]{2.5}},
\end{equation}
and
\begin{equation} \label{d}
d = \frac{2B}{b_2} - b_2 = \frac{2 b_1 (b_2 - b_1 + 1)}{2 b_1 - b_2} - b_2 = \frac{b_2^2 - 2 b_1^2 + 2 b_1}{2b_1 - b_2}
\end{equation}
Applying Lemma \ref{lem2} with \eqref{d} to \eqref{x3-y3-2}, the formula of $B$ in \eqref{A-B}, Lemma \ref{lem1} as $A - B \le 2 \sqrt{2} N^{1/4}$, and inequalities \eqref{lem0} and \eqref{b1}, we get
\begin{align*}
x_3 - y_3 \ge& 2 \sqrt{2 B} + \frac{(\frac{b_2^2 - 2 b_1^2 + 2 b_1}{2b_1 - b_2})^2}{8 B / b_2} = 2 \sqrt{2 B} + \frac{(b_2^2 - 2 b_1^2 + 2 b_1)^2}{8 (2b_1 - b_2) (b_2 - b_1 + 1)} \\
\ge& 2 \sqrt{2} \Bigl( \frac{B}{A} \Bigr)^{1/4} (A B)^{1/4} + \frac{(2 \sqrt{2} b_1^2 / N^{0.1})^2}{8 (2 - \sqrt{2}) b_1^2} \\
>& \Bigl( 2 \sqrt{2} - \frac{2}{N^{1/4}} - \frac{4 \sqrt{2}}{N^{1/2}} \Bigr) N^{1/4} + \frac{b_1^2}{(2 - \sqrt{2}) N^{0.2}} \\
>& 2 \sqrt{2} N^{1/4} - 2 - \frac{4 \sqrt{2}}{N^{1/4}} + \frac{N^{1/15}}{\sqrt[3]{6.25} (2 - \sqrt{2})} > 2 \sqrt{2} N^{1/4}
\end{align*}
when $N \ge 2^{20}$.

\bigskip

Now, we are going to construct examples to show that the lower bound is best possible. We pick $d_1 = 1$ and $d_2 = 2$ based on the above proof. Set $A - B = k$. From \eqref{dB}, we want
\begin{equation} \label{dB2}
B = b_1 (b_1 + k + 1) \; \; \text{ and } \; \; 2 B = b_2 (b_2 + k + 2).
\end{equation}
Pick $k = 2$ (other choices of $k$ would yield similar examples). Then
\begin{equation} \label{pell}
b_2 (b_2 + 4) = 2 b_1 (b_1 + 3) \; \; \text{ or } \; \; (2 b_1 + 3)^2 - 2 (b_2 + 2)^2 = 1.
\end{equation}
The Pell equation $x^2 - 2 y^2 = 1$ has infinitely many integer solutions given by $x_n + \sqrt{2} y_n = (3 + 2 \sqrt{2})^n$ and $x_n$ is always odd by considering the equation $(\bmod \, 4)$. Thus, we have
\[
b_1 = \frac{x_n - 3}{2}, \; b_2 = y_n - 2, \; a_1 = \frac{x_n - 1}{2}, \; a_2 = y_n, \; B = b_1 (b_1 + 3), \; A = (b_1 + 1)(b_1 + 2).
\]
One can check that
\[
B - b_1 = b_1 (b_1 + 2), \; \; A + a_1 = (b_1 + 1)(b_1 + 3),
\]
\[
\; B - b_2 = \frac{b_2 (b_2 + 2)}{2}, \; \; A + a_2 = \frac{(b_2 + 2)(b_2 + 4)}{2},
\]
and
\begin{align*}
(B - b_2)(A + a_2) =& \frac{b_2 (b_2 + 2)}{2} \cdot \frac{(b_2 + 2) (b_2 + 4)}{2} \\
=& \frac{(b_2 + 2)^2}{2} \cdot \frac{b_2 (b_2  +4)}{2} = b_1 (b_1 + 1) (b_1 + 2) (b_1  +3) = A B.  
\end{align*}
The ratio
\[
\frac{x_3 - y_3}{N^{1/4}} = \frac{A - B + 2 b_2 + d_2}{N^{1/4}} = \frac{2 (b_2 + 2)}{( b_1 (b_1 + 1) (b_1 + 2)(b_1 + 3) )^{1/4}} = 2 \sqrt{2} \Bigl( \frac{(b_1 + 1)(b_1 + 2)}{b_1 (b_1 + 3)} \Bigr)^{1/4}
\]
by \eqref{pell}. Note that
\[
\Bigl( \frac{(x + 1)(x + 2)}{x (x + 3)} \Bigr)^{1/4} = \Bigl(1 + \frac{2}{x (x + 3)} \Bigr)^{1/4} < 1 + \frac{2}{4 x (x+3)} < 1 + \frac{1}{\sqrt{x (x + 1) (x + 2) (x + 3)}}
\]
for $x \ge 1$. Hence,
\[
2 \sqrt{2} N^{1/4} < x_3 - y_3 < 2 \sqrt{2} N^{1/4} + \frac{2 \sqrt{2}}{N^{1/4}} \le 2 \sqrt{2} N^{1/4} + 0.1
\]
for $N \ge 2^{20}$. This shows that the lower bound $\lfloor 2 \sqrt{2} N^{1/4} + 1/2 \rfloor$ is best possible.
\end{proof}

\bigskip

{\it Remark:} The above examples come from solutions of Pell equation which has exponential growth. One can imitate the construction in Theorem \ref{thm2} to create polynomial examples with slightly larger constants. For instance, choose
\[
\left\{
\begin{array}{ll}
x_1 = 2K^2- 2 = 2(K + 1)(K - 1), & y_1 = 2K^2 - K = K (2 K - 1); \\
x_2 = 2K^2 + K - 1 = (2K - 1)(K + 1), & y_2 = 2 K^2 - 2 K = 2 K (K - 1); \\
x_3 = 2K^2 + 2K = 2K (K+1), & y_3 = 2 K^2 - 3K + 1 = (2K - 1)(K - 1);
\end{array}
\right.
\]
and
\[
N = x_1 y_1 = x_2 y_2 = x_3 y_3 = (K - 1) (K + 1) 2K (2K - 1).
\]
Then $x_3 > x_2 > x_1 > \sqrt{N}$ and $\max_{i} |x_i - y_i| = 5K - 1$ which is about $2.5 \sqrt{2} N^{1/4}$.


Department of Mathematics \\
Kennesaw State University \\
Marietta, GA 30060 \\
tchan4@kennesaw.edu

\end{document}